\documentclass{amsart}[12pt]
\usepackage{graphicx,graphics, amsmath, amsthm, amscd, amsfonts}

\usepackage{latexsym}
\makeatletter \oddsidemargin.9375in \evensidemargin \oddsidemargin
\newtheorem{theorem}{Theorem}[section]
\newtheorem{lemma}[theorem]{Lemma}

\theoremstyle{definition}

\numberwithin{equation}{section}

\begin{document}
\Large
\title[Maps preserving the fixed points of products of
operators]{Maps preserving the fixed points of products of operators}

\author[Ali Taghavi, Roja Hosseinzadeh and Vahid Darvish]{Ali Taghavi, Roja Hosseinzadeh and Vahid Darvish}

\address{{ Department of Mathematics, Faculty of Mathematical Sciences,
 University of Mazandaran, P. O. Box 47416-1468, Babolsar, Iran.}}

\email{taghavi@umz.ac.ir,  ro.hosseinzadeh@umz.ac.ir, v.darvish@stu.umz.ac.ir}



\begin{abstract}\large
Let $X$ be a complex Banach space with $\dim X\geq3$ and $B(X)$
the algebra of all bounded linear operators on $X$. Suppose
$\phi:B(X)\longrightarrow B(X)$ is a surjective map satisfying
the following property:
$$Fix(AB)=Fix(\phi(A)\phi(B)), \ \ \ (A, B\in B(X)).$$
Then the form of $\phi$ is characterized, where $Fix(T)$ is the
set of all fixed points of an operator $T$.
\end{abstract} \maketitle

\section{Introduction}
\noindent The study of maps on operator algebras
preserving certain properties is a topic which
attracts much attention of many authors. Some of these problems are concerned with preserving a certain property of usual products or other products of operators (see \cite{2}-\cite{5}, \cite{7}, \cite{9}-\cite{11} and \cite{13}-\cite{14}.)

Let $B(X)$ denotes the algebra of all bounded linear operators on a complex Banach space $X$ with $\dim X\geq3$. Recall that $x \in X$ is a fixed point of an operator $T \in B(X)$, whenever we have $Tx=x$. Denote by $Fix(T)$,
the set of all fixed points of $T$. In this paper, we characterize the concrete forms of surjective maps on $B(X)$ such that preserve the fixed points of products of operators.
 The statements of our main results are the follows.

\textbf{Main Theorem.} Let $X$ be a complex Banach space with $\dim
X\geq3$. Suppose $\phi:B(X)\longrightarrow B(X)$ is a surjective map with
the property that
\begin{equation}\label{*}
Fix(AB)=Fix(\phi(A)\phi(B)) \ \ \ \ \ (A\in B(X), B \in F_1(X)),
\end{equation}
then $\phi(A)=\kappa(A) A$, where $\kappa:B(X)\longrightarrow
\mathbb{C}$ is a function such that for any non-scalar operator
$A \in B(X)$ we have $\kappa (A)=1$ or $-1$ and for any scalar
operator $ \lambda I$ we have $ \kappa ( \lambda I)= \gamma (
\lambda)I$, where $\gamma:\mathbb{C}\longrightarrow\mathbb{C}$ is
a bijective map.

\section{The statement of the main results}
\noindent First we recall some
notations. $X^ *$
denotes the dual space of $X$. For every nonzero $x\in X$ and
$f\in X^ *$, the symbol $x\otimes f$ stands for the rank-one
linear operator on $X$ defined by $(x\otimes f)y=f(y)x$ for any
$y\in X$. Note that every rank-one operator in $B(X)$ can be
written in this way. The rank-one operator $x\otimes f$ is
idempotent if and only if $f(x)=1$ and is nilpotent if and only if $f(x)=0$. We denote by $F_1(X)$ and $P_1(X)$ the
set of all rank-one operators and the
set of all rank-one idempotent operators in $B(X)$, respectively.

Let $x\otimes f$ be a rank-one operator. It is easy to check that $x\otimes f$ is an idempotent if and only if $Fix(x\otimes f)=\langle x\rangle$ (the linear subspace spanned by $x$). If $x\otimes f$ isn't idempotent, then $Fix(x\otimes f)=\{0\}$.\\
First we prove some auxiliary lemmas. Let $\phi:B(X)\longrightarrow B(X)$ be a surjective map which satisfies condition (\ref{*})
\begin{lemma}\label{prop}
Let $A,B \in B(X)$. Then $A=B$ if and only if
$Fix(AT)=Fix(BT)$, for all $T\in F_{1}(X)$.
\end{lemma}
\begin{proof}
If $A\neq B$, then there is a $x\in X$ such that $Ax\neq Bx$. So
there exists a linear functional $f$ such that $f(Ax)=1$ and
$f(Bx)\neq1$. Setting $T=x\otimes f$, we have $Fix(Ax\otimes
f)=\langle Ax\rangle $ but $Fix(Bx\otimes f)=\{0\}$. This
contradiction yields assertion. The converse side is clear.
\end{proof}
\begin{lemma}\label{a}
$\phi (0)=0$ and $\phi$ is injective.
\end{lemma}
\begin{proof}
If $ \phi (0) \neq 0$, then there is a vector $x$ such that $
\phi (0)x \neq 0$, which implies that there is a linear
functional $f$ such that $f( \phi (0)x)=1$. Since $\phi$ is
surjective, there is an operator $A \in B(X)$ such that
$\phi(A)=x \otimes f$. By the preserving property of $ \phi$ we obtain $0=Fix(0.A)=Fix( \phi (0)x \otimes f)$. This is a contradiction, because $Fix( \phi (0)x \otimes f)= \langle \phi (0)x\rangle$. So $ \phi (0)=0$.\\
If $A,B\in B(X)$ such that $\phi(A)=\phi(B)$, then from
(\ref{*}) we obtain $Fix(AT)=Fix(BT)$ for every $T\in F_{1}(X)$ and this by
Lemma (\ref{prop}) yields $A=B$.
\end{proof}
\begin{lemma}\label{lemma1} $\phi$ preserves rank-one operators in both direction.
\end{lemma}
\begin{proof}
Let $\phi(x\otimes f)=T$ and $\mathrm{rank} T\geq 2$. So there exist linearly
independent $y_{1}, y_{2}$ in range of $T$. This implies that there exist $x_{1}, x_{2} \in X$ such that
$Tx_{1}=y_{1}$ and $Tx_{2}=y_{2}$. We can find a bounded linear operator $S$
such that $Sy_{1}=x_{1}$ and $Sy_{2}=x_{2}$. Hence
$$STx_{1}=Sy_{1}=x_{1},$$
$$STx_{2}=Sy_{2}=x_{2},$$
these imply
\begin{equation}\label{eq1}
x_{1},x_{2}\in Fix(ST).
\end{equation}
It is easy to see that $x_{1}$ and $x_{2}$ are linearly
independent. On the other hand, since $\phi$ is surjective, there exists an operator $S^{'}\in B(X)$ such
that $\phi(S^{'})=S$. So we have,
$$Fix(ST)=Fix(S^{'}x\otimes f)\subseteq \langle S^{'}x\rangle.$$
This is a contradiction, since $Fix(ST)$ contains at least two linearly
independent elements. Therefore, $\mathrm{rank} T\leq 1$. This together with Lemma (\ref{a}) implies that $\mathrm{rank} T=1$.\\
For the converse, we can write
$Fix(AB)=Fix(\phi^{-1}(A)\phi^{-1}(B))$, because  $\phi$ is
bijective. By a similar way as above, $\phi^{-1}$ preserves rank-one
operators, and this completes the proof.
\end{proof}
\begin{lemma}\label{lemmab} $\phi (I)=I$ or $-I$.
\end{lemma}
\begin{proof}
Let $x\otimes f$ be a rank-one idempotent operator. By lemma (\ref{lemma1}) there exist $y \in X$ and $g\in X^*$ such that $\phi(x\otimes f)=y\otimes g$. We have,
$$Fix(x\otimes f)=Fix((x\otimes f)^{2})=Fix((y\otimes
g)^{2})=Fix(g(y)y\otimes g),$$
 we obtain $g(y)=1$ or $-1$, which implies that $\phi$ or $-\phi$ preserves the rank-one idempotents in both directions.\\
 Let $\phi (I)=U$. If $\phi$ preserves the rank-one idempotents in both directions and $y \in X$ and $g \in X^*$ such that $g(y)=1$, then from
 $$Fix(x\otimes f)=Fix(Uy\otimes g),$$
 we obtain $g(Uy)=1$. Therefore, for any $y \in X$ and $g \in X^*$ such that $g(y)=1$, we have $g(Uy)=1$. Now it is easy to check that $U=I$. If $-\phi$ preserves the rank-one idempotents in both directions, with a similar discussion we obtain $U=-I$. The proof is complete.
\end{proof}
Without loss of generality, next we assume that $\phi (I)=I$.
\begin{lemma}\label{lemma2}
$\phi(P)=P$, for all $P \in P_1(X)$.
\end{lemma}
\begin{proof}
Let $x\otimes f$ be a rank-one operator. By lemma (\ref{lemma1}) there exist $y \in X$ and $g\in X^*$ such that
$\phi(x\otimes f)=y\otimes g$.  From (\ref{lemmab}) and the preserving property of $\phi$ we obtain,
$$Fix(x\otimes f)=Fix(y\otimes g)$$
which implies $\langle x\rangle=\langle y\rangle$ and so $x$ and
$y$ are linearly dependent. Without loss of generality, we assume
that $\phi(x\otimes f)=x\otimes g$. From
$$Fix(x\otimes f)=Fix((x\otimes f)^{2})=Fix((x\otimes
g)^{2})=Fix(g(x)x\otimes g),$$
 we obtain $g(x)=1$ or $-1$ and so $x \not \in \ker g$. Thus we have
\begin{equation*}
\ker f\cup\langle x\rangle=X,
\end{equation*}
\begin{equation}\label{eqn1}
\ker g\cup\langle x\rangle=X.
\end{equation}
On the other hand, it is clear that,
\begin{equation*}
\ker f\cap\langle x\rangle=\{0\},
\end{equation*}
\begin{equation}\label{eqn2}
\ker g\cap\langle x\rangle=\{0\}.
\end{equation}
From (\ref{eqn1}) and (\ref{eqn2}) we have $\ker f=\ker g$, this
means $g$ is multiple of $f$. Since $x\otimes f$ and $x\otimes g$
or $x\otimes f$ and $-x\otimes g$ are both idempotent, we obtain
$f=g$ and this completes the proof.
\end{proof}
\begin{lemma}\label{lemma4}
Let $A,B\in B(X)$ be non-scalar operators. If $Fix(AR)=Fix(BR)$ for
every $R\in P_{1}(X)$, then $B=\lambda I+(1-\lambda)A$ for some
$\lambda\in\mathbb{C}\setminus\{1\}$.
\end{lemma}
\begin{proof}
Suppose $AR\in P_{1}(X)\setminus\{0\}$. Since $0\neq
Fix(AR)=Fix(BR)$ and $BR$ is rank-one,$BR\in
P_{1}(X)\setminus\{0\}$. By proposition 2.3 in \cite{5}, there is
a $\lambda\neq1$ such that $B=\lambda I+(1-\lambda)A$.
\end{proof}
\noindent In the following lemma we denote
$\mathbb{C}^{*}=\mathbb{C}\setminus\{0,1\}$.

\begin{lemma}\label{lemma5}
Let $A \in B(X)$. Then $A\in \mathbb{C}^{*}I$
if and only if $Fix(AR)=\{0\}$, for every $R \in P_1(X)$.
\end{lemma}
\begin{proof}
If $A\notin \mathbb{C}^{*}I$, then there exists a vector $x \in X$ such
that $x$ and $Ax$ are linearly independent. This implies that $x$ and $Ax-x$ are linearly independent, too.
So there is a linear functional $f$ such that $f(x)=1$ and $f(Ax-x)=0$. Setting $R=x\otimes f$, we obtain $Fix(AR)=\langle x\rangle$, a contradiction. Hence $A\in \mathbb{C}^{*}I$. The other side is clear.
\end{proof}
\textbf{Proof of Main Theorem.} We divide the proof of main theorem into two cases.\\
 \text{Case1}: $\phi(\lambda I)=\gamma(\lambda)I$, where
$\gamma:\mathbb{C}\longrightarrow\mathbb{C}$ is a bijective map.\\
If $ \lambda \in\mathbb{C}^{*}$, from Lemmas (\ref{lemma2}) and (\ref{lemma5}) we obtain
\begin{eqnarray*}
A\in\mathbb{C}^{*}I&\Leftrightarrow&\forall R\in P_{1}(X):
Fix(AR)=\{0\},\\
&\Leftrightarrow& \forall R\in P_{1}(X): Fix(\phi(A)R)=\{0\},\\
&\Leftrightarrow& \phi(A)\in \mathbb{C}^{*}I.
\end{eqnarray*}
By Lemmas (\ref{a}) and (\ref{lemmab}), $ \phi (0)=0$ and $ \phi (I)= \pm I$ and this completes the proof.\\
\text{Case2}: $\phi(A)=A$ or $-A$, for every non-scalar operator $A \in B(X)$.\\
Suppose $N=x\otimes f$ is a rank-one nilpotent operator. By Lemma
\ref{lemma4}, there exists a $\lambda\in\mathbb{C}\setminus\{1\}$
such that $\phi(N)=\lambda I+(1-\lambda)N$. By Lemma
\ref{lemma1}, since $\phi(N)$ is a rank-one operator, we obtain
$\lambda=0$ and so
$\phi(N)=N$. \\
Now let $A$ is an arbitrary non-scalar operator. By Lemma
\ref{lemma4}, there exists a $\lambda\in\mathbb{C}\setminus\{1\}$
such that $\phi(A)=\lambda I+(1-\lambda)A$. On the other hand,
there exists a vector $x$ such that $x$ and $Ax$ are linearly
independent. So there is a linear functional  $f$ such that
$f(x)=0$ and $f(Ax)=1$. We have
$$Fix(AN)=Fix((\lambda I+(1-\lambda)A)N)=Fix(\lambda N+(1-\lambda)AN).$$
Since $Ax \in Fix(AN)=Fix(Ax\otimes f)$, $Ax \in Fix(\lambda N+(1-\lambda)AN)=Fix(\lambda x\otimes f+(1-\lambda)Ax\otimes f)$ and so
$$(\lambda x\otimes f+(1-\lambda)Ax\otimes f)Ax=Ax$$
which imply $\lambda x+(1-\lambda)Ax=Ax$ and so $\lambda=0$, because $x$ and
$Ax$ are linearly independent. Therefore, $\phi(A)=A$.
\par \vspace{.4cm}{\bf Acknowledgements:} This research is partially
supported by the Research Center in Algebraic Hyperstructures and
Fuzzy Mathematics, University of Mazandaran, Babolsar, Iran.

\bibliographystyle{amsplain}

\end{document}